\documentclass[12pt]{article}
\usepackage{amssymb,amsthm}
\usepackage{amsmath}
\usepackage{amscd}
\usepackage{youngtab}
\usepackage{tikz}
\usepackage{colortbl}
\usepackage{tikz-cd}
\usepackage[all]{xy}
\newtheorem{thm}{Theorem}
\newtheorem{lema}[thm]{Lemma}
\newtheorem{cor}[thm]{Corollary}
\newtheorem{prop}[thm]{Proposition} 

\newtheorem{defi}[thm]{Definition}
\newtheorem{exa}[thm]{Example}

\definecolor{amber}{rgb}{1.0, 0.75, 0.0} 
\definecolor{americanrose}{rgb}{1.0, 0.01, 0.24}
\definecolor{amber(sae/ece)}{rgb}{1.0, 0.49, 0.0}

\begin{document}
\setlength{\baselineskip}{16pt}
\title{On multi-symmetric functions and transportation polytopes}
\author{Eddy Pariguan and Jhoan Sierra V}
\maketitle
\begin{abstract}
We present a study of the transportation polytopes appearing in the product rule of elementary multi-symmetric functions introduced by F. Vaccarino. \end{abstract}

\section*{Introduction}
The classical transportation problems in operation research arise from the problem of transporting goods from a set of factories, and a set of consumer centers. Assuming the total supply of the set of factories equals to the total demand of consumer centers, we can optimize the cost of transporting goods (see \cite{Hitch, Hant, Koop}). Transportation polytopes have an interest in discrete mathematics and also arise naturally in optimization and statistics (see \cite{Hoff, Motz, Neu, Yem}). 

 A transportation polytope consists of all tables of non-negative real numbers that satisfy certain equations. In this work  we only consider the well-known subfamily, the classical transportation polytopes in just two indices, the $2$-way transportation polytopes and we use the notation and terminology introduced by Jesus A. De Loera and Edward D. Kim in \cite{Loe}.

Our main motivation comes from the study of the product rule of elementary multi-symmetric functions introduced by F. Vaccarino in \cite{Vac}  and their relationships with transportation polytopes.  The classic product rule of multi-symmetric functions  and its respective generalization to the quantum case introduced by Diaz and Pariguan in \cite{DP}, both have an unexplored underlying structure of transportation polytopes. The main goal of this work, see section \ref{LS},  is to present a first combinatorial description of this structure in the classical case.

\section{Review of multi-symmetric  functions}

In this section, we present a short introduction to elementary multi-symmetric functions. Fix a characteristic zero field $\mathbb{K}$. Consider the action of the symmetric group $S_n$ on ${\mathbb{K}}^n$ by permutation of vector entries. The quotient space ${\mathbb{K}}^n /S_n$ is the configuration space of $n$-unlabeled points with repetitions in $\mathbb{K}$. Polynomials functions on ${\mathbb{K}}^n /S_n$ may be identified with the algebra $\mathbb{K}[x_1,\cdots,x_n]^{S_n}$ of $S_n$ invariant polynomials in $\mathbb{K}[x_1,\cdots,x_n]$. It is well-known that ${\mathbb{K}}^n /S_n$ is an $n$-dimensional affine space; indeed we have an isomorphism of algebras
$$\mathbb{K}[x_1,\cdots,x_n]^{S_n} \equiv \mathbb{K}[e_1,\cdots,e_n], $$
where $\alpha\in[n]=\{1,2,\cdots,n\}$ and $e_{\alpha}$ is the elementary symmetric polynomial determined by the identity
$$\prod_{i=1}^n(1+x_it)=\sum_{\alpha=0}^n e_{\alpha}(x_1,\cdots, x_n)t^{\alpha}.$$

If we consider polynomial functions over $(\mathbb{K}^d)^n/S_n$,
we obtain the ring of multi-symmetric functions, also called the ring of vector symmetric functions or MacMahon's symmetric functions \cite{MacD}, which are given by
$$\mathbb{K}[x_{11},\cdots,x_{1d},x_{22},\cdots,x_{2d},\cdots,x_{n1},\cdots,x_{nd}]^{S_n}.$$

We will denote by $\mathbb{K}[(\mathbb{K}^d)^n]^{S_n}$ to the ring $\mathbb{K}[x_{11},\cdots,x_{nd}]^{S_n}$. The following results due to F. Vaccarino.

Fix $p,n,d\in\mathbb{N}^+$. Let $y_1,\cdots,y_d$ and $t_1,\cdots,t_d$ be independent and commutative variables in $\mathbb{K}$. For $\alpha=(\alpha_1,\cdots,\alpha_p)\in\mathbb{N}^p$ we use the following notation
$$|\alpha|=\displaystyle\sum_{i=1}^{p}\alpha_i,\hspace{1em}t^{\alpha}=\prod_{i=1}^{p}t_i^{\alpha_i}.$$

Given a polynomial $f\in\mathbb{K}[y_1,\cdots,y_d]$ and $i\in[n]$, we denote by $f(i)=f(x_{i1},\cdots,x_{id})$ to the polynomial obtained by replacing each appearance of $y_j$ in $f$ by  $x_{ij}$, for $j\in[d]$.

\begin{defi}
	Fix $\alpha\in\mathbb{N}^p$ such that $|\alpha|\leq n$ and $f=(f_1,\cdots,f_p)\in\mathbb{K}[y_1,\cdots,y_d]^p$. The multisymmetric functions $e_\alpha(f)\in\mathbb{K}[(\mathbb{K}^d)^n]^{S_n}$, are given by the identity
	$$\prod_{i=1}^{n}(1+f_1(i)t_1+f_2(i)t_2+\cdots+f_p(i)t_p)=\displaystyle\sum_{|\alpha|\leq n}e_\alpha(f)t^\alpha.$$
\end{defi}

The following result provide an explicit formula for the product rule of multi-symmetric functions
	\begin{thm}\label{teoimportante}
		Fix $p,q,n\in\mathbb{N}^+$, $f\in\mathbb{K}[y_1,\cdots,y_d]^p$ and $g\in\mathbb{K}[y_1,\cdots,y_d]^q$. Let $\alpha\in\mathbb{N}^p$ and $\beta\in\mathbb{N}^q$ be such that $|\alpha|,|\beta|\leq n$, then we have
		$$e_\alpha(f) e_\beta(g)=\displaystyle\sum_{\gamma\in L(\alpha,\beta,n)}e_\gamma(f,g,fg),$$
		where:
		\begin{enumerate}
			\item $(f,g,fg)=(f_1,\cdots,f_p,g_1,\cdots,g_q,f_1g_1,\cdots,f_1g_q,f_2g_1,\cdots,f_2g_q,\cdots,f_pg_1,\cdots,f_pg_q)$.
			\item $L(\alpha,\beta,n)$ is the set of matrices $\gamma\in{\rm{Map}}(\{0\}\cup[p]\times\{0\}\cup[q], \mathbb{N})$ such that
			\begin{itemize}
				\item $\gamma_{00}=0$,
				\item $|\gamma|=\displaystyle\sum_{i=0}^{p}\displaystyle\sum_{j=0}^{q}\gamma_{ij}\leq n$,
				\item $\displaystyle\sum_{j=1}^{q}\gamma_{ij}=\alpha_i$ for $i\in [p]$. 
				\item $\displaystyle\sum_{i=1}^{p}\gamma_{ij}=\beta_j$ for $j\in [q]$.
			\end{itemize}
		\end{enumerate}
	\end{thm}

Graphically, a matrix $\gamma$ is represented as
$$
\begin{array}{cc}
\begin{array}{ c c c c c c c c}
\phantom{abc} & \beta_{1\phantom{a}}  & \beta_{2\phantom{a}} & \beta_{3\phantom{a}} & \mbox{} & \beta_{q\phantom{a}}  \\
\mbox{} & \uparrow & \uparrow& \uparrow& \cdots  & \uparrow\\
\end{array} &  \begin{array}{cc}
\mbox{}  & \mbox{} \\
\mbox{}  & \mbox{} \\
\end{array}\\
\begin{bmatrix}
0          &\gamma_{01}&\gamma_{02}&\gamma_{03}&\cdots&\gamma_{0q}\\
\gamma_{10}&\gamma_{11}&\gamma_{12}&\gamma_{13}&\cdots&\gamma_{1q}\\
\gamma_{20}&\gamma_{21}&\gamma_{22}&\gamma_{23}&\cdots&\gamma_{2q}\\
\vdots & \mbox{} & \mbox{} & \mbox{} & \mbox{} & \vdots\\
\gamma_{p0}&\gamma_{p1}&\gamma_{p2}&\gamma_{p3}&\cdots&\gamma_{pq}
\end{bmatrix} & \begin{array}{cc}
\mbox{} & \mbox{}\\
\to & \alpha_1\\
\to & \alpha_2\\
\mbox{} & \vdots\\
\to & \alpha_p
\end{array} 
\end{array}$$

where the arrows $\to\uparrow$ represent, respectively, row and column sums and the matrix $\gamma$ will be identify with the vector
$$\vec{\gamma}=(\gamma_{10},\cdots,\gamma_{p0},\gamma_{01},\cdots,\gamma_{0q},\gamma_{11},\cdots,\gamma_{1q},\gamma_{21},\cdots,\gamma_{2q},\cdots,\gamma_{p1},\cdots,\gamma_{pq}).$$

The main goal of this work is the study of the combinatorial structure underlying in the set of matrices $L(\alpha,\beta,n)$ introduced in Theorem \ref{teoimportante}.

\begin{exa}\label{ejemlucas2}
	For $n=3,\alpha=(2,1),\beta=(1,2),f=(y_1,y_2)$ and $g=(y_1y_3,y_2)$, we have the following identity
	
	$$e_{(2,1)}(y_1,y_2)e_{(1,2)}(y_1y_2,y_3)=\displaystyle\sum_{\gamma}e_\gamma(y_1,y_2,y_1y_2y_3,y_1^2y_3,y_1y_2y_3,y_1y_2,y_2^2)$$
	where $\gamma=(\gamma_{10},\gamma_{20},\gamma_{01},\gamma_{02},\gamma_{11},\gamma_{12},\gamma_{21},\gamma_{22})\in\mathbb{N}^8$ is such that $|\gamma|\leq3$ and:\\
	
	\begin{minipage}{.5\linewidth}
		\begin{center}
			$\gamma_{10}+\gamma_{11}+\gamma_{12}=2,$\\
			$\gamma_{20}+\gamma_{21}+\gamma_{22}=1,$
		\end{center}
	\end{minipage}
	\begin{minipage}{.5\linewidth}
		\begin{center}
			$\gamma_{01}+\gamma_{11}+\gamma_{21}=1,$\\
			$\gamma_{20}+\gamma_{12}+\gamma_{22}=2.$
		\end{center}
	\end{minipage}\\
	
Finding the solutions we obtain the vectors
$$(0,0,0,0,1,1,0,1),(0,0,0,0,0,2,1,0)$$
then we have that
$$\begin{array}{lcl}
e_{(2,1)}(y_1,y_2)e_{(1,2)}(y_1y_3,y_2) & = & e_{(1,1,1)}(y_1^2y_3,y_1y_2y_3,y_2^2)\\
\mbox{} & + & e_{(2,1)}(y_1^2y_3,y_1y_2y_3,y_2^2).
\end{array}$$
\end{exa}

\section{Classical transportation polytopes }
In this section, we review a few needed notions on classical $2$-way transportation polytopes and we assume the
reader to be somewhat familiar with De Loera and Kim's work \cite{Loe}. 

\begin{defi}\label{politransporte}
	Fix $p,q\in\mathbb{N}$ and let $u\in\mathbb{R}^p_{\geq0}$,  $v\in\mathbb{R}^q_{\geq0}$  be two vectors. The transportation polytope $P$ of size $p\times q$ defined by the vectors $u$ and $v$ is the convex polytope on $p\times q$ variables $x_{ij}\in\mathbb{R}_{\geq0}$, where $i\in[p]$ and $j\in[q]$, which satisfy the $p+q$ equations given by:
	\begin{equation}\label{polycon}
	\sum_{j=1}^{q}x_{ij}=u_i\text{ and } \displaystyle\sum_{i=1}^{p}x_{ij}=v_j. 
	\end{equation}
The vectors $u$ and $v$ are called marginals vectors or margins vectors of the polytope $P$.
\end{defi}
These polytopes are called transportation polytopes because they model the transportation of goods from $p$ supply locations to $q$ demand locations. 

\begin{exa}\label{ejpolyt}
Let us consider the transportation of goods for $3$-supply locations to  $3$-demand location with suppliyng vector  $u=(5,4,3)$ and 
demanding vector $v=(6,2,4)$.  A point $x$ in the transportation polytope $P$ of size $3\times 3$ defined by the margins $u$ and $v$ is given by
$$\begin{array}{cccc}
\mbox{} & \mbox{} & \mbox{} & \mbox{} \\
\mbox{} & \mbox{} & \mbox{} & \mbox{} \\
x&=&
\begin{bmatrix}
x_{11}&x_{12}&x_{13}\\
x_{21}&x_{22}&x_{23}\\
x_{31}&x_{32}&x_{33}\\
\end{bmatrix}
&=
\end{array}
\begin{array}{cc}
\begin{array}{ c c c c c}
5 & 4 & 3 \\
\uparrow & \uparrow & \uparrow\\
\end{array} &  \begin{array}{cc}
\mbox{}  & \mbox{} \\
\mbox{}  & \mbox{} \\
\end{array}\\
\begin{bmatrix}
4 & 1 & 1\\
0 & 2 & 0\\
1 & 1 & 2\\ 
\end{bmatrix} & \begin{array}{cc}
\to & 6\\
\to & 2\\
\to & 4\\
\end{array} 
\end{array}$$
where the horizontal and vertical arrows represent, respectively, row and column sums.

\end{exa}

\begin{lema}\label{lemma5}
	Let $P$ be a $2$-way transportation polytope of size $p\times q$ defined by the margins $u\in\mathbb{R}^p_{\geq0}$ and $v\in\mathbb{R}^q_{\geq0}$. 
	The polytope $P$ is not empty is and only if
	$$\displaystyle\sum_{i\in[p]}u_i=\displaystyle\sum_{j\in[q]}v_j.$$
\end{lema}

This proof uses the northwest corner rule algorithm (see \cite{Tp}).

The equations given in (\ref{polycon}) and the inequalities $x_{ij} \geq 0$ can be expressed in matrix form as follows
	\begin{equation}\label{MatrixR} P=\{x\in\mathbb{R}^{pq}:Ax=b,x\geq0\},\end{equation}
where $A$ is a matrix of size $(p+q)\times pq$ and $b\in\mathbb{R}^{p+q}$. The matrix $A$ is called the constraint matrix.

Transportation polytopes have a relationship with complete bipartite graph $K_{p,q}$ (\cite{Sal, Graph}) of two sets of vertices of $U$ and $V$ of cardinality $p$ and $q$, respectively, when we consider $U$ the supply and $V$ is the demand.

\begin{defi}
The graph $K_{p,q}$ is the complete bipartite graph consisting of two sets $U$ and $V$ of cardinality $p$ and $q$, respectively such that for any $i\in U$ and $j\in V$ there is an edge $e_{ij}$ connecting them. 
\end{defi}
It is well known that the constraint matrix for a $p\times q$ transportation polytope is the vertex-edge incidence matrix of the complete bipartite graph $K_{p,q}$.

\begin{exa}
Consider the $3\times 3$ transportation polytope $P$ defined by $u=(5,4,3)$ and $v=(6,2,4)$, then the complete bipartite graph $K_{3,3}$ is given by:\\
	
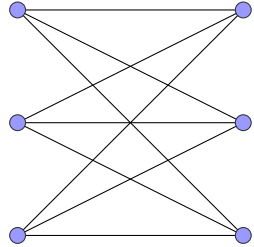
	\begin{figure}[h!]
		\centering
		\begin{tikzpicture}[y=.3cm, x=.3cm,font=\normalsize]
		
		\draw (0,0)  -- (10,0);
		\draw (0,0)  -- (10,5);
		\draw (0,0)  -- (10,10);
		\draw (0,5)  -- (10,0);
		\draw (0,5)  -- (10,5);
		\draw (0,5)  -- (10,10);
		\draw (0,10) -- (10,0);
		\draw (0,10) -- (10,5);
		\draw (0,10) -- (10,10);
		
		\filldraw[fill=blue!40,draw=black!80] (0,0) circle (3pt);
		\filldraw[fill=blue!40,draw=black!80] (0,5) circle (3pt);
		\filldraw[fill=blue!40,draw=black!80] (0,10) circle (3pt);
		\filldraw[fill=blue!40,draw=black!80] (10,0) circle (3pt);
		\filldraw[fill=blue!40,draw=black!80] (10,5) circle (3pt);
		\filldraw[fill=blue!40,draw=black!80] (10,10) circle (3pt);
		\end{tikzpicture}
	\end{figure}

 We also have that  $P=\{x\in\mathbb{R}^9:Ax=b,x \geq 0\}$, where the constraint matrix $A$ is given as follows
	$$A=\begin{bmatrix}
	1 & 0 & 0 & 1 & 0 & 0 & 1 & 0 & 0\\
	0 & 1 & 0 & 0 & 1 & 0 & 0 & 1 & 0\\
	0 & 0 & 1 & 0 & 0 & 1 & 0 & 0 & 1\\
	1 & 1 & 1 & 0 & 0 & 0 & 0 & 0 & 0\\
	0 & 0 & 0 & 1 & 1 & 1 & 0 & 0 & 0\\
	0 & 0 & 0 & 0 & 0 & 0 & 1 & 1 & 1
	\end{bmatrix}
	\textrm{ and\ \ \ } b=
	\begin{bmatrix}
	5\\
	4\\
	3\\
	6\\
	2\\
	4\\
	\end{bmatrix}.$$
	
In the Example \ref{ejpolyt}, the solution of $Ax=b$ can be expressed as $x^t=(4,1,1,0,2,0,1,1,2)$.

\end{exa}

\section{Multi-symmetric functions and transportation polytopes}\label{LS}

The product rule of elementary multi-symmetric functions given in Theorem \ref{teoimportante} involve a set of matrices with some remarkable properties. In this section we will provide some characterizations of the set $L(\alpha,\beta,n)$ in terms of transportation polytopes. In order to simplify our notation we will denote by  $L$ to the set $L(\alpha,\beta,n)$ (see Definition \ref{L}) and we can think of $\gamma\in L$ as natural points of transportation polytopes $P$.

In particular, the study of integer points of transportation polytopes is very popular in combinatorics, a lot of mathematical objects rich in combinatorial properties appear when we study integer points in polytopes such as magic squares \cite{Beck}, sudoku arrangements \cite{Stanley}, and others.

\begin{defi}\label{PN}
	Fix $p,q,N\in\mathbb{N}$ and let $u\in\mathbb{N}^{p+1}$,  $v\in\mathbb{N}^{q+1}$  be two vectors such that ${\displaystyle u_0=N-\sum_{i=1}^p u_i}$ and $\displaystyle{v_0=N-\sum_{i=1}^q v_i}$. The transportation polytope $P_N$ of size $p+1\times q+1$ defined by the vectors $u$ and $v$ is the convex polytope on $p+1\times q+1$ variables $x_{ij}\in\mathbb{R}_{\geq0}$, where $i\in\{0\}\cup[p]$ and $j\in\{0\}\cup[q]$, which satisfy the $p+q+2$ equations given by:
	\begin{equation}\label{polyconN}
	\sum_{j=0}^{q}x_{ij}=u_i\text{ and } \displaystyle\sum_{i=0}^{p}x_{ij}=v_j. 
	\end{equation}
\end{defi}

\begin{defi}\label{L}Fix $p,q,n\in\mathbb{N}$,  $\alpha\in\mathbb{N}^p$ and $\beta\in\mathbb{N}^q$. We denote by $L$ the set of matrices  $\gamma\in{\rm{Map}}(\{0\}\cup[p]\times\{0\}\cup[q], \mathbb{N})$ which satisfy the equations	\begin{itemize}
		\item $\gamma_{00}=0$.
		\item $|\gamma|=\displaystyle\sum_{i=0}^{p}\displaystyle\sum_{j=0}^{q}\gamma_{ij}\leq n.$
		\item $\displaystyle\sum_{j=1}^{q}\gamma_{ij}=\alpha_{i}$ for $i\in [p]$ 
		\item $\displaystyle\sum_{i=1}^{p}\gamma_{ij}=\beta_j$ for $j\in [q]$.
	\end{itemize}
\end{defi}

\begin{exa}
	For $\alpha=(2,1),\beta=(1,2)$ and $n=3$, the set $L$ is given by:
	$$L=\left\{\begin{bmatrix}
	0&0&0\\
	0&1&1\\
	0&0&1
	\end{bmatrix},
	\begin{bmatrix}
	0&0&0\\
	0&0&2\\
	0&1&0
	\end{bmatrix}\right\}.$$
\end{exa}


We denote by $L_N$ the subset of $L$ given by:
\begin{equation} L_N=\{\gamma\in L: |\gamma_{ij}|=N, \ \mbox{for some}\ N\leq n\} \end{equation}
The following result provides some combinatorial properties of $L_N$.

\begin{thm}\label{teo1}  The following identities holds 
\begin{enumerate}
\item{$L_N\neq \emptyset$ if $ \max\{|\alpha|,|\beta|\}\leq N \leq |\alpha|+|\beta|$.}
\item{ ${\displaystyle L=\bigsqcup_{N=\max\{|\alpha|,|\beta|\}}^{n}L_N}$, if $n< |\alpha|+|\beta|$.}  
\item{ ${\displaystyle L=\bigsqcup_{N=\max\{|\alpha|,|\beta|\}}^{|\alpha|+|\beta| }L_N}$, if $n\geq |\alpha|+|\beta|$.}  
\end{enumerate}
\end{thm}
\begin{proof}
Fix $N$ such that $ \max\{|\alpha|,|\beta|\}\leq N \leq |\alpha|+|\beta|$. We are going to construct an element $\gamma$ such that  $\gamma \in L_N$ as follows: Let $\gamma\in L$  such that $(\gamma_{01},\gamma_{02},\cdots,\gamma_{0q})$ be a $q$-weak composition of $\alpha_0$ which satisfy $\gamma_{0j}\leq \beta_j$ $\forall j\in[q]$, and let $(\gamma_{10},\gamma_{20},\cdots,\gamma_{p0})$ be a $p$-weak composition of $\beta_0$  which satisfy $\gamma_{i0}\leq\alpha_i$ $\forall i\in[p]$.

Denote by $\beta_{j}^{(k)}:=\beta_j-{\displaystyle \sum_{i=1}^{k-1}\gamma_{ij}}$, for $(k,j)\in[p]\times [q]$ and let $(\gamma_{11},\gamma_{12},\cdots,\gamma_{1q})$ be a $q$-weak composition of $\alpha_1-\gamma_{10}$ which satisfy $\gamma_{1j}\leq\beta_j^{(1)}$. 
Analogously we consider $(\gamma_{21},\gamma_{22},\cdots,\gamma_{2q})$ a $q$-weak composition of $\alpha_2-\gamma_{20}$ such that $\gamma_{2j}\leq\beta_j^{(2)}$.
Let's go through this process until we get $(\gamma_{p1},\gamma_{p2},\cdots,\gamma_{pq})$ a $q$-weak composition of $\alpha_p-\gamma_{p0}$ with $\gamma_{pj}\leq\beta_j^{(p)}$ and finally under this construction the reader can check that
  $\gamma_{ij}=\gamma\in L_N$, therefore $L_N\neq \emptyset$.

It is not difficult to check statements $2$ and $3$.
\end{proof}

\begin{exa}\label{ejem1}
The set $L$ defined by vectors $\alpha=(1,1)$, $\beta=(2,1)$ and $n=4$ is given by
	$$L=\left\{\begin{bmatrix}
	0&0&1\\
	0&1&0\\
	0&1&0
	\end{bmatrix},
	\begin{bmatrix}
	0&1&0\\
	0&0&1\\
	0&1&0
	\end{bmatrix},
	\begin{bmatrix}
	0&1&0\\
	0&1&0\\
	0&0&1
	\end{bmatrix},
	\begin{bmatrix}
	0&2&0\\
	0&0&1\\
	1&0&0
	\end{bmatrix},
	\begin{bmatrix}
	0&2&0\\
	1&0&0\\
	0&0&1
	\end{bmatrix},
	\begin{bmatrix}
	0&1&1\\
	1&0&0\\
	0&1&0
	\end{bmatrix},
	\begin{bmatrix}
	0&1&1\\
	0&1&0\\
	1&0&0
	\end{bmatrix}\right\}.$$
We have that  ${\displaystyle L=\bigsqcup_{N=3}^4L_N}$,  where $L_3$ and $L_{4}$ are given by
	$$L_3=\left\{\begin{bmatrix}
	0&0&1\\
	0&1&0\\
	0&1&0
	\end{bmatrix},
	\begin{bmatrix}
	0&1&0\\
	0&0&1\\
	0&1&0
	\end{bmatrix},
	\begin{bmatrix}
	0&1&0\\
	0&1&0\\
	0&0&1
	\end{bmatrix}\right\},$$

and

	$$\phantom{bla}L_4=\left\{
	\begin{bmatrix}
	0&2&0\\
	0&0&1\\
	1&0&0
	\end{bmatrix},
	\begin{bmatrix}
	0&2&0\\
	1&0&0\\
	0&0&1
	\end{bmatrix},
	\begin{bmatrix}
	0&1&1\\
	1&0&0\\
	0&1&0
	\end{bmatrix},
	\begin{bmatrix}
	0&1&1\\
	0&1&0\\
	1&0&0
	\end{bmatrix}\right\}.$$
	
\end{exa}
The following result shows that $L_N$ is a set of natural points in some transportation polytope.
\begin{thm}\label{Poly} There is a transportation polytope $P_M$ such that $L_N\subset P_M$.
\end{thm}
\begin{proof} Let $\gamma \in L_N$, then $\gamma$ satisfy the equations given in Definition \ref{L}.  Under the assumptions of Definition \ref{PN}, consider the transportation polytope $P_M$  defined by margins $\overline{\alpha}\in\mathbb{N}^{p+1} $ and $\overline{\beta}\in\mathbb{N}^{q+1}$ such that 
\begin{itemize}
\item{$ \overline{\alpha}=(\alpha_0,\alpha_1,\cdots,\alpha_p)$ with $\alpha_0=M-{\displaystyle \sum_{i=1}^p \alpha_i}$. }
\item{$\overline{\beta}=(\beta_0,\beta_{1},\cdots,\beta_{q})$ with $\beta_0=M-{\displaystyle \sum_{i=1}^q \beta_i}$.}
\end{itemize}
It should be clear that $\gamma \in P_M$ if $M=N$.
\end{proof}

We make a few remarks regarding to Theorem \ref{Poly}. Elements $\gamma\in L_N$ are such that $|\gamma|=N$ and $\gamma\in L(\alpha, \beta,n)=L$, hence for $p,q,n\in \mathbb{N}^+$,   $\alpha\in\mathbb{N}^p$ and $\beta\in\mathbb{N}^q$,  $\gamma$ satisfy the conditions of Theorem \ref{teoimportante}. To find the transportation polytope $P_M$ such that $L_N\subset P_M$, we consider the transportation polytope defined by margins $ \overline{\alpha}$ and $\overline{\beta}$ which are obtained from $\alpha$ and $\beta$ adding new inputs $\alpha_0$, $\beta_{0}$  satisfying the condition given above. We stress that we will work with the transportation polytope $P_N$ which follows this previous construction. \\
This previous considerations imply our next result which establishes an example of the transportation polytopes associated with sets $L_3$ and $L_4$ given in Example \ref{ejem1}.

\begin{exa}\label{ejemploprinc}
Fix $p=q=2$,  $N=3$ and consider the vectors $\overline{\alpha}=(1,1,1)$, $\overline{\beta}=(0,2,1)$.  The transportation polytope 
$P_3$  defined by margins ${\overline{\alpha}}, {\overline{\beta}}$ is given by 	
	$$P_3=\left\{X\in M_{3\times 3}(\mathbb{R}_{\geq0}):\displaystyle \sum_{j=1}^3x_{ij}=\overline{\alpha}_i\text{ and } \sum_{i=1}^3x_{ij}=\overline{\beta}_j\right\},$$ 
and we have $L_3\subset P_3$. If we consider $X=\begin{bmatrix}
	0&\frac{1}{2}&\frac{1}{2}\\
	0&\frac{1}{2}&\frac{1}{2}\\
	0&1&0
	\end{bmatrix}$ 	
we have that $X\in P_3$ but $x\notin L_3$ and therefore $L_3\neq P_3$.\\

On the other hand,  fix $p=q=2$,  $N=4$ and consider the vectors  $\overline{\alpha}=(2,1,1)$, $\overline{\beta}=(1,2,1)$. The transportation polytope 
$P_4$  defined by margins ${\overline{\alpha}}, {\overline{\beta}}$ is given by 
	$$P_4=\left\{X\in M_{3\times 3}(\mathbb{R}_{\geq0}):\displaystyle \sum_{j=1}^3x_{ij}=\overline{\alpha}_i\text{ and } \sum_{i=1}^3x_{ij}=\overline{\beta}_j\right\},$$ 
and we have $L_4\subset P_4$. If we consider	 $X=\begin{bmatrix}
	\frac{1}{2}&1&\frac{1}{2}\\
	\frac{1}{2}&0&\frac{1}{2}\\
	0&1&0
	\end{bmatrix}$ 
we have that $X\in P_4$ but $X\notin L_4$ and therefore $L_4\neq P_4$.
\end{exa}

It is well known that transportation polytopes $P$ can be represented in matrix form, therefore transportation polytopes $P_N$ can be represented in matrix form as well (see Proposition \ref{teo2}).  In this case we consider the graph $K'_{p,q}$  obtained from $K_{p,q}$ removing the edge $e_{11}$. Figure \ref{k'33} shows the graph $K'_{3,3}$ associated to $K_{3,3}$.

\begin{figure}[h!]
	\centering
	\begin{tikzpicture}[y=.3cm, x=.3cm,font=\normalsize]
	
	\draw (0,0)  -- (10,0);
	\draw (0,0)  -- (10,5);
	\draw (0,0)  -- (10,10);
	\draw (0,5)  -- (10,0);
	\draw (0,5)  -- (10,5);
	\draw (0,5)  -- (10,10);
	\draw (0,10) -- (10,0);
	\draw (0,10) -- (10,5);
	
	\filldraw[fill=blue!40,draw=black!80] (0,0) circle (3pt);
	\filldraw[fill=blue!40,draw=black!80] (0,5) circle (3pt);
	\filldraw[fill=blue!40,draw=black!80] (0,10) circle (3pt);
	\filldraw[fill=blue!40,draw=black!80] (10,0) circle (3pt);
	\filldraw[fill=blue!40,draw=black!80] (10,5) circle (3pt);
	\filldraw[fill=blue!40,draw=black!80] (10,10) circle (3pt);
	\end{tikzpicture}
	\caption{$K^{\prime}_{3,3}$ graph.}\label{k'33}
\end{figure}


The following result provides the matrix form associated to $L_N$.
	\begin{prop}\label{teo2}
	For any $N\in\mathbb{N}$, each $L_N$ can be expressed as follows:
		$$L_N=\{x_N\in\mathbb{N}^{(p+1)(q+1)-1}: Ax_N=b_N\},$$
where $b_N=(\alpha_0,\alpha_1,\cdots,\alpha_p,\beta_0,\beta_{1},\cdots,\beta_{q})$ is such that $\alpha_0=N-{\displaystyle \sum_{i=1}^p \alpha_i}$,  $\beta_0=N-{\displaystyle \sum_{i=1}^q \beta_i}$, and $A$ is the matrix is obtain by following the next construction  
\begin{enumerate}
\item{Let $B$ be the constraint matrix of  $K^{\prime}_{p+1,q+1}$, and denote by $B^i$ the $i$-th column of $B$, for all $i$.}
\item{For $i\in[p]$ the $i$-th column $A^i$ of matrix $A$ is given by $A^i=B^{i(q+1)}$.}
\item{For $i\in[q]$ we have $A^{p+i}=B^{i}$.}
\item{Last columns of $A$ are obtained from $B$ after rearranging in ascended way the remaining columns. }
\end{enumerate}
\end{prop}

\begin{proof}
Let $P_N$ be the transportation polytope such that $L_N\subset P_N$.
It should be clear that $P_N$ is an special case of $2$-way transportation polytope for any $N\in \mathbb{N}$. Observe that for $\gamma\in P_N$ we have $\gamma_{00}=0$, then $P_N$ can be expressed in matrix form as follows (see equation (\ref{MatrixR}))
$$P_N=\{x_N\in\mathbb{R}_{\geq0}^{(p+1)(q+1)-1}:Bx_N=b_N\},$$
where $B$ is the constraint matrix of the graph $K^{\prime}_{{p+1,q+1}}$. 

The matrix $A$ obtained from $B$ following the previous construction provides a rearrangement of $x_N$ such that solutions of the equation $Ax_N=b_N$ are vectors 
$\vec{\gamma}$ which satisfy the conditions of Theorem \ref{teoimportante}, therefore we have the desired result.

\end{proof}

\begin{exa}
For $N=3$ and $b_N=(1,1,1,0,2,1)$, we have that
	$$L_3=\{x_3\in\mathbb{N}^8:Ax_N=b_3\},$$
where $A$ is given as follows :\\

Let $B$ be the constraint matrix of  $K^{\prime}_{3,3}$ given by:

	$$B=\left[\begin{array}{cccccccc}
	1&1&0&0&0&0&0&0\\
	0&0&1&1&1&0&0&0\\
	0&0&0&0&0&1&1&1\\
	0&0&1&0&0&1&0&0\\
	1&0&0&1&0&0&1&0\\
	0&1&0&0&1&0&0&1
	\end{array}\right].$$
	
Under the assumptions of Proposition \ref{teo2}, for $p=q=2$, we have that

\begin{itemize}
	\item $A^1=B^3$ and $A^2=B^6$,
	\item $A^3=B^1$ and $A^4=B^2$,
	\item The last four columns of $A$ are given by $A^5=B^4,  A^6=B^5, A^7=B^7$ and $A^8=B^8$.
\end{itemize}
Then we have
	

$$\begin{array}{ccccr}
B=\begin{array}{cc}
\begin{array}{l @{\hspace{2mm}}l@{\hspace{2mm}}l@{\hspace{2mm}}l@{\hspace{2mm}}l@{\hspace{2mm}}l@{\hspace{2mm}}l@{\hspace{2mm}}l@{\hspace{2mm}}}
\scriptstyle{B^{1}}&\scriptstyle{B^{2}}&\scriptstyle{B^{3}}&\scriptstyle{B^{4}}&\scriptstyle{B^{5}}&\scriptstyle{B^{6}}&\scriptstyle{B^{7}}&\scriptstyle{B^{8}}\\
\end{array}\\
\begin{bmatrix}
1&1&0&0&0&0&0&0\\
0&0&1&1&1&0&0&0\\
0&0&0&0&0&1&1&1\\
0&0&1&0&0&1&0&0\\
1&0&0&1&0&0&1&0\\
0&1&0&0&1&0&0&1
\end{bmatrix}
\end{array} & \longrightarrow & A=
\end{array}
\begin{array}{cc}
\begin{array}{l @{\hspace{2mm}}l@{\hspace{2mm}}l@{\hspace{2mm}}l@{\hspace{2mm}}l@{\hspace{2mm}}l@{\hspace{2mm}}l@{\hspace{2mm}}l@{\hspace{2mm}}}
\scriptstyle{B^{3}}&\scriptstyle{B^{6}}&\scriptstyle{B^{1}}&\scriptstyle{B^{2}}&\scriptstyle{B^{4}}&\scriptstyle{B^{5}}&\scriptstyle{B^{7}}&\scriptstyle{B^{8}}\\
\end{array}\\
\begin{bmatrix}
0&0&1&1&0&0&0&0\\
1&0&0&0&1&1&0&0\\
0&1&0&0&0&0&1&1\\
1&1&0&0&0&0&0&0\\
0&0&1&0&1&0&1&0\\
0&0&0&1&0&1&0&1
\end{bmatrix}
\end{array}.$$
\end{exa}

Our next goal is to describe the structure of $\mathbb{N}$-matrix of the elements of $L_N$. To accomplish it, we will require some definitions due to R. Stanley  (see \cite{Stanley2}).
Let $A=(a_{ij})$ be  an $\mathbb{N}$-matrix with finitely many nonzero entries, that is $A$ is an $\mathbb{N}$-matrix of finite support and we can think of $A$ as either an infinity matrix or as an $m\times n$ matrix  when $a_{ij}=0$ for $i>m$ and $j>n$. Associate with $A$ a generalized permutation or two-line array $\omega_A$ given by
$$\omega_A=\left(\begin{matrix} i_1&i_2&i_3&\cdots&i_{m} \\ j_1&j_2&j_3&\cdots&j_{m} \end{matrix}\right)$$ such that 1. $i_1\leq i_2\leq\cdots\leq i_m$, 2. if $i_r=i_s$ and $r\leq s$ then $j_r\leq j_s$, and 3. for each pair $(i,j)$, there are exactly $a_{ij}$ values of $r$ for which $(i_r,j_r)=(i,j)$. $A$ determines a unique two-line array $\omega_{A}$ satisfying this conditions and conversely any such array corresponds to a unique $A$. For instance, if
 $A=\left (\begin{matrix} 0&0&1\\0&1&0\\0&1&0\end{matrix}\right)$, then the corresponding two-line array is
 $\omega_A=\left(\begin{matrix} 1&2&3 \\ 3&2&2 \end{matrix}\right).$

\begin{defi}
Fix $A$ an $\mathbb{N}$-matrix and let $\omega_{A}$ be the two-line array associate with $A$. We denote by ${\rm{type}}^1(\omega_{A})$  the vector $(u_1,\cdots,u_{m})$ such that 
the natural number $k$ appears exactly $u_k$ times in the first row of $\omega_A$ and we denote by ${\rm{type}}^2(\omega_{A})$  the vector $(v_1,\cdots,v_{m})$  such that the natural number $k$ appears exactly $v_k$ times in the second row of $\omega_A$.
\end{defi}
\begin{exa}
If
$\omega_A=\left(\begin{matrix} 1&2&3 \\ 3&2&2 \end{matrix}\right),$ then  ${\rm{type}}^1(\omega_A)=(1,1,1)$ and ${\rm{type}}^2(\omega_A)=(0,2,1).$
\end{exa}

Fix $N\in\mathbb{N}$, we denote by $\omega_N$ the set of two-line array given by
$$\omega_N=\left\{\omega_A=\left(\begin{matrix} i_1&i_2&i_3&\cdots&i_{N} \\ j_1&j_2&j_3&\cdots&j_{N} \end{matrix}\right): (i_1,j_1)\neq (1,1)\right\}$$

\begin{thm}\label{bijm}
There is a bijection between elements of $L_N$ and elements $\omega_A\in\omega_N$ such that  ${\rm{type}}^1(\omega_{A})=\overline{\alpha}$ and ${\rm{type}}^2(\omega_{A})=\overline{\beta}$.
\end{thm} 

\begin{proof}
Let $\gamma\in L_N$. Using Stanley's construction, we can think of $\gamma\in L_N$ as an $(p+1)\times (q+1)$ matrix  when $\gamma_{ij}=0$ for $i>p+1$ and $j>q+1$, then $\gamma$ determines a unique two-line array $\omega_{\gamma}$ satisfying the previous conditions. It should be clear that there is a injective map $L_N\to \omega_N$. On the other hand, note that since the elements $\omega_A\in\omega_N$ are such that $(i_1,j_1)\neq (1,1)$, it follows that $a_{11}=0$, moreover  ${\rm{type}}^1(\omega_{A})=\overline{\alpha}$ and ${\rm{type}}^2(\omega_{A})=\overline{\beta}$, then $A=(a_{ij})$ is an element of $L_N$. Thus we conclude that there is a injective map $\omega_N\to L_N$.

\end{proof}

\begin{exa} Under the assumptions of Example \ref{ejem1} consider $L_4$  given by 

$$\phantom{bla}L_4=\left\{
	\begin{bmatrix}
	0&2&0\\
	0&0&1\\
	1&0&0
	\end{bmatrix},
	\begin{bmatrix}
	0&2&0\\
	1&0&0\\
	0&0&1
	\end{bmatrix},
	\begin{bmatrix}
	0&1&1\\
	1&0&0\\
	0&1&0
	\end{bmatrix},
	\begin{bmatrix}
	0&1&1\\
	0&1&0\\
	1&0&0
	\end{bmatrix}\right\}.$$
	
The set $\omega_4$ is given by 
$$\omega_4=\left \{  \left(\begin{matrix} 1&1& 2&3 \\ 2&2&3&1 \end{matrix}\right), \left(\begin{matrix} 1&1& 2&3 \\ 2&2&1&3 \end{matrix}\right), \left(\begin{matrix} 1&1& 2&3 \\ 2&3&1&2 \end{matrix}\right), \left(\begin{matrix} 1&1& 2&3 \\ 2&3&2&1 \end{matrix}\right)
 \right \}. $$
\end{exa}

It is well known that we can associated with an $\mathbb{N}$-matrix $A$ of finite support a pair $(P,Q)$ of semistandard Young tableau  (SSYT) of the same shape using the RSK algorithm. The RSK algorithm is a bijection between $\mathbb{N}$-matrices of finite support and ordered pairs $(P,Q)$ of SSYTs of the same shape.

On the other hand, we know that any $\gamma\in L_N$ is an $\mathbb{N}$-matrix of finite support such that  ${\rm{row}}(\gamma)=\overline{\alpha}$ and  ${\rm{col}}(\gamma)=\overline{\beta}$. Using Theorem \ref{bijm} we can see that RSK algorithm is a bijection between elements $\gamma\in L_N$ and ordered pairs $(P,Q)$ of SSYTs of the same shape such that 
${\rm{type}}(P)={\rm{col}}({\gamma})=\overline{\beta}$,  ${\rm{type}}(Q)={\rm{row}}(\gamma)=\overline{\alpha}$ and the first box of the last row of $P$ and $Q$ is not equal to $1$ simultaneously. Therefore, we can summarize it as follows

\begin{cor}
There is a bijection between $L_N$ and ordered pairs $(P,Q)$ of SSYTs of the same shape such that ${\rm{type}}(P)={\rm{col}}({\gamma})=\overline{\beta}$,  ${\rm{type}}(Q)={\rm{row}}(\gamma)=\overline{\alpha}$ and the first box of the last row of $P$ and $Q$ is not equal to  $1$ simultaneously.
\end{cor}

\begin{exa} Let $\omega_\gamma=\left(\begin{matrix} 1&2&3 \\ 3&2&2 \end{matrix}\right)$ be the two-line array associated with 
 $\gamma=\left (\begin{matrix} 0&0&1\\0&1&0\\0&1&0\end{matrix}\right)$. The ordered pairs $(P,Q)$ of SSYTs are the following
 $$\left (\young(3,22), \young(2,13)\right)$$

\end{exa}

\subsection*{Acknowledgments} 
We thank the organizers of "Encuentro Colombiano de Combinatoria (ECCO)". We
are also grateful to Carolina Benedetti, Rafael Díaz, Rafael González and Felipe Rincón for their suggestions.  This work has been supported by Pontificia Universidad Javeriana.

\noindent epariguan@javeriana.edu.co\\
\noindent Departamento de Matem\'aticas. Pontificia Universidad Javeriana. Bogot\'a, Colombia\\

\noindent jhoan.sierra@utalca.cl\\
\noindent Instituto de Matemáticas y Física. Universidad de Talca. Talca. Chile.\\

\end{document}